\documentclass[11pt]{article}

\usepackage{amssymb,amsfonts,amsmath, cite}
\usepackage{enumerate}

\usepackage[english]{babel}
\usepackage[cp1250]{inputenc}

\usepackage{algorithm2e}

\usepackage{longtable}

\usepackage{graphicx}
\usepackage{index}
\usepackage{fancyhdr}
\usepackage{lhelp}
\usepackage{optparams}
\usepackage{psfrag}
\usepackage{forloop}
\usepackage{setspace}
\usepackage{tikz}
\usepackage{subcaption}
\usepackage{pgffor}
\usepackage{xifthen}

\definecolor{lgray}{gray}{0.75}

\usetikzlibrary{decorations.pathreplacing,automata,calc,positioning}

\newcommand{\pch}{\chi_{\rho}}
\newcommand{\diam}{{\rm diam}}
\newcommand{\rad}{{\rm rad}}
\newcommand{\C}{{\cal C}}

\newcommand{\qed}{\hfill $\square$ \bigskip}

\newcommand{\mptt}[1]{}

\newtheorem{theorem}{Theorem} 
\newtheorem{corollary}[theorem]{Corollary}
\newtheorem{lemma}[theorem]{Lemma}
\newtheorem{proposition}[theorem]{Proposition}

\begin{document}

\title{Graphs that are critical for the packing chromatic number}

\author{
Bo\v{s}tjan Bre\v{s}ar$^{a,b}$  \and Jasmina Ferme$^{c,a}$ 
 }

\date{}

\maketitle

\begin{center}
$^a$ Faculty of Natural Sciences and Mathematics, University of Maribor, Slovenia\\
\medskip

$^b$ Institute of Mathematics, Physics and Mechanics, Ljubljana, Slovenia\\
\medskip

$^c$ Faculty of Education, University of Maribor, Slovenia\\

\end{center}

\begin{abstract}
Given a graph $G$, a coloring $c:V(G)\longrightarrow \{1,\ldots,k\}$ such that $c(u)=c(v)=i$ implies that vertices $u$ and $v$ are at distance greater than $i$, is called a packing coloring of $G$. The minimum number of colors in a packing coloring of $G$ is called the packing chromatic number of $G$, and is denoted by $\pch(G)$. In this paper, we propose the study of $\pch$-critical graphs, which are the graphs $G$ such that for any proper subgraph $H$ of $G$, $\pch(H)<\pch(G)$. We characterize $\pch$-critical graphs with diameter 2, and $\pch$-critical block graphs with diameter 3. Furthermore, we characterize $\pch$-critical graphs with small packing chromatic numbers, and we also consider $\pch$-critical trees. In addition, we prove that for any graph $G$ with $e\in E(G)$, we have $(\pch(G)+1)/2\le \pch(G-e)\le \pch(G)$, and provide a corresponding realization result, which shows that $\pch(G-e)$ can achieve any of the integers between the bounds. 
\end{abstract}


\medskip\noindent
{\bf AMS Subj.\ Class:} 05C15, 05C70, 05C12

\section{Introduction}
The packing chromatic number was introduced by Goddard et al.~\cite{goddard-2008} in 2008 under the name broadcast chromatic number, and has been investigated by a number of authors. The wide interest given to this problem is reflected in the (probably non-exhaustive) list of papers on this problem that were published only in the last two years~\cite{balogh-2018, balogh-2019, bf-2018a, gt-2019, bf-2018b, bkrw-2018, klmp-2018, kr-2019, kv-2018,kv-2019, ls-2018}. A lot of attention was given to the question of boundedness of this invariant in the class of cubic graphs, which was recently answered in the negative by Balogh, Kostochka and Liu~\cite{balogh-2018}, see also an explicit construction~\cite{bf-2018b}. Several other classes of (finite and infinite) graphs were studied for their packing chromatic number~\cite{ekstein-2014, lb-2017,shao-2015,togni-2014,torres-2015}, among which the largest attention was probably given to the infinite square grid; the last paper in a series due to Barnaby et al.~\cite{barnaby-2017} limits the packing chromatic number of the infinite square grid to be in the set $\{13,14,15\}$. The decision version of the packing chromatic number is NP-complete, and it remains NP-complete, somewhat surprisingly, even in trees, as proven by Fiala and Golovach~\cite{fiala-2010} (see also a more recent investigation on the complexity of the packing coloring problem~\cite{klmp-2018}).

It is clear that the invariant is hereditary in the sense that a graph cannot have smaller packing chromatic number that its subgraphs. The behaviour of the invariant under some local operations, such as edge-contraction, vertex- and edge-deletion, and edge subdivision was studied in~\cite{bkrw-2017}. In particular, it was shown that the difference between $\pch(G)$ and $\pch(G-e)$, where $e$ is an edge in $G$, can be arbitrarily large, and the same also holds for vertex-deleted subgraphs. Recently, Klav\v zar and Rall~\cite{kr-2019} investigated the class of graphs $G$ for which $\pch(G-v)<\pch(G)$ for every vertex $v\in V(G)$, called packing chromatic vertex-critical graphs. Among other results, they characterized packing chromatic vertex-critical graphs $G$ with small values of $\pch(G)$, and also presented several properties of trees, which are packing chromatic vertex-critical. 

In this paper, we study another (basic) version of critical graphs for the packing chromatic number, where it is required that $\pch(H)<\pch(G)$ for each proper subgraph $H$ of a graph $G$. Note that in graphs $G$ with no isolated vertices this is equivalent to the statement that $\chi_\rho(G-e)<\chi_\rho(G)$ for every edge $e$ of $G$. (We remark that the authors in~\cite{kr-2019} suggested to study $\pch$-edge-critical graphs, meaning of which should be clear.)
Critical graphs for the standard chromatic number present a classical notion in graph theory, see e.g.~\cite{west}. It turns out (as seen in this paper) that there are not many similarities between ($\chi$-)critical graphs and $\pch$-critical graphs. On the other hand, in studying $\pch$-critical graphs, several results about $\pch$-vertex critical graphs can be used (see Section~\ref{sec:vertex}).

The paper is organized as follows. In the next section, we establish the notation and define the concepts used throughout the paper. Then, we prove a lower bound for the packing chromatic number of the edge-deleted graph, which reads as $\chi_\rho(G-e) \geq \frac{\chi_\rho(G)+1}{2}$. Not only that the bound is sharp, we also prove that for any integer $n\in [\frac{k+1}{2}, k]$, there exists a graph $G$ with an edge $e$ such that $\chi_\rho(G)=k$ and $\chi_\rho(G-e)=n$. In Section~\ref{sec:vertex}, we use some results from~\cite{kr-2019} concerning $\pch$-vertex-critical graphs, and establish characterizations of $\pch$-critical graphs $G$ when $\pch(G)\in\{2,3\}$. We also give some partial results when $\pch(G)=4$, and prove that $\pch$-critical and $\pch$-vertex-critical trees coincide. In Section~\ref{sec:diam2}, we characterize the packing chromatic critical graphs with diameter 2 as the graphs in which for every edge $e\in E(G)$ one of the two properties hold: the independence number of $G-e$ is bigger than that of $G$, or $\diam(G-e)>2$ and there is a maximum independent set of $G$ that avoids two vertices at distance at least 3 in $G-e$ (one of which is an endvertex of $e$ and the other a neighbor of an endvertex of $e$). Finally, in Section~\ref{sec:block} we give a structural characterization of $\pch$-critical block graphs with diameter 3, which divides these graphs into three different types. The proof of this result is quite involved, which indicates the difficulty of the study of $\pch$-critical graphs. We end the paper with some remarks and open questions.


\section{Notation and preliminaries}
\label{sec:preliminaries}

In this paper, we consider only finite, simple graphs. Let $G$ be a graph and $v$ an arbitrary vertex of $G$. The {\em(open) neighborhood} of $v$, denoted by $N_G(v)$, is the set of all vertices adjacent to $v$. The {\em degree} of $v$, $\deg_G(v)$,  is $|N_G(v)|$. If $\deg_G(v)=1$, then $v$ is a {\em leaf}. A vertex $u \in V(G)$ adjacent to at least one leaf is a {\em support vertex} and its {\em leaf neighbor} is a leaf adjacent to $u$. Next, if $N_G(v)$ induces a complete graph, then $v$ is called a {\em simplicial vertex}. The minimum degree of vertices in a graph $G$ is denoted by $\delta(G)$. We will also use the notation $N_G[v]$, which presents the {\em closed neighborhood} of $v$, where $N_G[v]=N_G(v) \cup \{v\}$.
The {\em eccentricity} of a vertex $v$, denoted by $\epsilon_G(v)$, is the maximum distance between $v$ and any other vertex of $G$: $\epsilon_G(v)= \max_{u\in V(G)}\{d(v,u)\}$. 
The subscript in the above notations may be omitted if the graph $G$ is clear from context. 
The {\em radius} of $G$, $\rad(G)$, and the {\em diameter} of $G$, $\diam(G)$, are the minimum and the maximum eccentricities. The {\em center} of $G$, $C(G)$, is the set of vertices with minimum eccentricity, that is: $C(G) = \{u \in V (G)\,|\,\epsilon_G(u) =\rad(G)\}$. 

The distance, $d_G(u,v)$ (sometimes abbreviated to $d(u,v)$), between two vertices $u$ and $v$ of graph $G$ is the length of a shortest $u,v$-path in $G$. 
Given a graph $G$ and a positive integer $i$, an {\em $i$-packing} in $G$ is a subset $W$ of the vertex set of $G$ such that the distance between any two distinct vertices from $W$ is greater than $i$. Note that 1-packing is equivalent to the concept of independent set. 
The {\em packing chromatic number} of $G$, denoted by $\pch(G)$, is the smallest integer $k$ such that the vertex set of $G$ can be partitioned into sets $V_1,\ldots, V_k$, where $V_i$ is an $i$-packing for each $i\in \{1,\ldots, k\}$.
The corresponding mapping $c:V(G)\longrightarrow [k]$ having the property that $c(u)=c(v)=i$ implies $d(u, v) > i$, is a {\em $k$-packing coloring}. When $k=\chi_\rho(G)$, we say that $k$-packing coloring is {\em optimal}.

A graph $G$ is a {\em packing chromatic critical graph}, or shorter {\em$\chi_\rho$-critical graph}, if for every proper subgraph $H$ of $G$, $\pch(H)<\pch(G)$.
In the case when $G$ is $\chi_\rho$-critical and $\chi_\rho(G)=k$ we also say that $G$ is {\em $k$-$\chi_\rho$-critical}. If for a given graph $G$ and for each $x \in V(G)$ we have $\chi_\rho(G-x)< \chi_\rho(G)$, then $G$ is called {\em $\chi_\rho$-vertex-critical graph}, and we also say $G$ is {$k$-$\chi_\rho$-vertex-critical} if $\pch(G)=k$. 

The hereditary behaviour of the packing chromatic number (notably, $\chi_\rho(H) \leq \chi_\rho(G)$ for any subgraph $H$ of a graph $G$) implies that $\chi_\rho(G-e) \leq \chi_\rho(G)$ for any edge $e$ of $G$.  
It was shown in~\cite{bkrw-2017} that for every positive integer $r$ there exists a graph $G$ with an edge $e$ such that $\pch(G)-\pch(G-e)\ge r$. Nevertheless, one can bound $\pch(G-e)$ from below by a fraction of $\pch(G)$. More precisely, the packing chromatic number cannot drop by more than one half when an edge is deleted.

\begin{theorem}
Let $G$ be a graph and $e$ an arbitrary edge of $G$. Then $\chi_\rho(G-e) \geq \frac{\chi_\rho(G)+1}{2}$ and the bound is sharp.
\label{izrek1}
\end{theorem}
\begin{proof}
Let $G$ be a graph, $e=uv$ an edge of $G$ and $c'$ any optimal packing coloring of $G-e$. In order to form a packing coloring $c$ of $G$ which uses at most $2 \cdot \chi_\rho(G-e)-1$ colors, we first let $c(w)=c'(w)$ for all $w \in V(G)$.
Since $d_G(x,y) \leq d_{G-e}(x,y)$ for any $x,y \in V(G)$, coloring $c$ may not be a proper packing coloring of $G$. If it is, then $\chi_\rho(G) \leq \chi_\rho(G-e) \leq 2\cdot \chi_\rho(G-e)-1$ and we are done. Otherwise, we need to ``correct'' the coloring $c$ by changing the colors of some vertices. The problematic vertices for a given color $k$, $k \in \{1, 2, \ldots, \chi_\rho(G-e)\}$, are those to which $c$ assigns the same color $k$ and are at distance at most $k$ in $G$. We observe that if $c(x)=c(y)=k$ and $d_G(x,y)\leq k$ for some $x,y \in V(G)$, then every shortest $x,y$-path in $G$ goes through the edge $e$, since $d_G(x,y) < d_{G-e}(x,y)$.

Let $k$ be a color used by $c$. We claim that for at least one vertex from $\{u,v\}$, say $u$, there exists at most one problematic vertex $y$ for a color $k$, which satisfies the property $d_{G}(y,u) \leq d_{G}(y,v)$. Suppose to the contrary that there exist vertices $x_1, x_2, y_1, y_2 \in V(G)$ that receive the same color $k$ by $c$ and satisfy the following properties: $d_G(x_i,u)\leq d_G(x_i,v)$ and $d_G(y_i,v) \leq d_G(y_i,u)$ for each $i \in \{1,2\}$. These two properties imply that there exist a shortest $x_1x_2$-path which does not contain the vertex $v$ (and thus does not contain the edge $e$), and a shortest $y_1y_2$-path, which does not contain the vertex $u$ (and the edge $e$). Hence, $d_{G}(x_1,x_2) \geq k+1$ and $d_{G}(y_1,y_2) \geq k+1$. Without loss of generality we may assume that $d_{G}(x_1,u)=a \geq d_{G}(x_2,u)=b$, $d_{G}(y_1,v)=c \geq d_{G}(y_2,v)=d$ and $a \geq c$.  Therefore, $a \geq \frac{k+1}{2}$ and $c \geq \frac{k+1}{2}$. If $d_{G}(x_1,y_1) \leq k$, then each of the shortest $x_1y_1$-paths contains the edge $e$ and thus $d_{G}(x_1,y_1)=a+1+c \geq k+2$, a contradiction. If $d_G(x_1,y_2) \geq k+1$, we are done (since $x_1$ is then not a problematic vertex for color $k$), thus suppose that $d_G(x_1,y_2) \leq k$. Then, any shortest $x_1y_2$-path contains the edge $e$ and we infer $d_G(x_1,y_2)=a+1+d \geq a+1+(k+1-c) \geq k+2$, which yields the claimed assertion. 

Therefore, when dealing with the problematic vertices for color $k$ with respect to coloring $c$, we only need to consider the case when there are vertices $x_1, \ldots, x_m$, $m \geq 1$, and a vertex $x$ such that $c(x_i)=c(x)=k$, $d_G(x_i,x) \leq k$ for $i\in [m]$, while $d_G(x_i,x_j) \geq k+1$ for all $1\le i<j\le m$. Then, by replacing the color $c(x)$ with the color $\chi_\rho(G-e)+k$, we provide the coloring in which all the vertices of $G$ colored with color $k$ are pairwise at distance at least $k+1$ in $G$ (note that the vertex $x$ is the only vertex of $G$, which receives color $\chi_\rho(G-e)+k$ by $c$). 
Using the same subsitution for each color $k \in \{1, 2, \ldots, \chi_\rho(G-e)\}$, we infer that at most $\chi_\rho(G-e)$ additional colors are used, and so $c$ becomes a proper packing coloring of $G$. Hence, $\chi_\rho(G) \leq 2\cdot \chi_\rho(G-e)$.

To prove the assertion stated in the theorem, we note that the only possibility for two vertices of $G$ receiving color $1$ by $c$ and being at distance $1$ in $G$ is that $c(u)=c(v)=1$. However, in this case $c$ cannot give color $2$ to two vertices, which are at distance at most $2$ in $G$. Indeed, any shortest path between such two vertices should contain $e=uv$ which makes the distance between them in $G$ greater than $2$. In either case, we need at most $\chi_\rho(G-e)-1$ additional colors (since either for color 1 or color 2, coloring $c$ does not need a substitution of colors), hence $\chi_\rho(G-e) \geq \frac{\chi_\rho(G)+1}{2}$. 

To see that the bound is sharp, consider the family of graphs $G_n$, $n \geq 2$, defined as follows. Let $A$ and $B$ be two copies of the graph $K_n$, $n \geq 2$, and let $a \in V (A)$, $b \in V (B)$. The graph $G_n$ is obtained from the disjoint union of $A$ and $B$ by adding an edge connecting the vertices $a$ and $b$, and attaching $2n-2$ leaves to each vertex in $(V(A) \cup V(B)) \setminus \{a,b\}$ (see Fig.~\ref{fig:Ex1} depicting graph $G_4$).
First we claim that $\chi_\rho(G_n)=2n-1$ for any $n \geq 2$. Let $c$ be any packing coloring of $G_n$ and suppose that it uses at most $2n-2$ colors. If $c(a')=1$ for some $a' \in V(A) \setminus \{a\}$, then all leaves adjacent to $a'$ receive different colors from $\{2, 3, \ldots, 2n-1\}$, a contradiction. Thus, $c(a') \neq 1 $ for all $a' \in V(A) \setminus \{a\}$ and analogously we derive that $c(b') \neq 1 $ for all $b' \in V(B) \setminus \{b\}$. Hence, $c$ can assign  color $1$ only to one vertex from $V(A) \cup V(B)$, namely only to $a$ or only to $b$. 
Since $A$ and $B$ are isomorphic to $K_n$, only two vertices from $V(A) \cup V(B)$ can be colored by color $2$. Thus, using the facts that $|V(A) \cup V(B)|=2n$, and that any color greater than 2 can appear only on one vertex from $V(A) \cup V(B)$, we infer that $\chi_\rho(G_n) \geq 2n-1$. By letting $c(a)=1$, $c(l)=1$ for all leaves $l$ of $G_n$, $c(a')=c(b')=2$ for some $a' \in V(A) \setminus \{a\}$, $b' \in V(B) \setminus \{b\}$, and by assigning $2n-3$ different colors from $\{3, 4, \ldots, 2n-1\}$ to other $2n-3$ vertices of $G_n$, we infer that $c$ is a packing coloring of $G_n$ using $2n-1$ colors, which implies that $\chi_\rho(G_n)=2n-1$ holds for any $n \geq 2$. 
Next, let $e=ab$. Note that $G_n-e$ consists of two connected components $A'$ and $B'$, which are isomorphic. Clearly, $\chi_\rho(A')\geq n$, since $A'$ contains a subgraph isomorphic to $K_n$. By letting $c'(a)=1$, $c'(l)=1$ for all leaves $l$ of $A'$, and by assigning the colors from $\{2, \ldots, n\}$ to the other vertices of $A'$ we infer that $\chi_\rho(A')=n$ and so $\chi_\rho(B')=n$. Since $\chi_\rho(G_n-e) =\chi_\rho(A')=n$, we infer that $\chi_\rho(G_n-e) = \frac{\chi_\rho(G_n)+1}{2}$ for any $n\ge 2$. 
\qed
\end{proof}


\begin{figure}[ht!]
\begin{center}
\begin{tikzpicture}[scale=1.0,style=thick]
\def\vr{2pt} 
\path (-1,-1) coordinate (a3); \path (-2,0) coordinate (a2);
\path (-1,1) coordinate (a1); \path (0,0) coordinate (a);
\path (2,-1) coordinate (b3); \path (3,0) coordinate (b2);
\path (2,1) coordinate (b1); \path (1,0) coordinate (b);

\path (-1.7,-1.6) coordinate (a31);
\path (-1.5,-1.9) coordinate (a32);
\path (-1.2,-2) coordinate (a33);
\path (-0.8,-2) coordinate (a34);
\path (-0.5,-1.9) coordinate (a35);
\path (-0.3,-1.6) coordinate (a36);
\path (-1.7,1.6) coordinate (a11);
\path (-1.5,1.9) coordinate (a12);
\path (-1.2,2) coordinate (a13);
\path (-0.8,2) coordinate (a14);
\path (-0.5,1.9) coordinate (a15);
\path (-0.3,1.6) coordinate (a16);
\path (-2.6,0.7) coordinate (a21);
\path (-2.9,0.5) coordinate (a22);
\path (-3,0.2) coordinate (a23);
\path (-3,-0.2) coordinate (a24);
\path (-2.9,-0.5) coordinate (a25);
\path (-2.6,-0.7) coordinate (a26);

\path (2.7,-1.6) coordinate (b31);
\path (2.5,-1.9) coordinate (b32);
\path (2.2,-2) coordinate (b33);
\path (1.8,-2) coordinate (b34);
\path (1.5,-1.9) coordinate (b35);
\path (1.3,-1.6) coordinate (b36);
\path (2.7,1.6) coordinate (b11);
\path (2.5,1.9) coordinate (b12);
\path (2.2,2) coordinate (b13);
\path (1.8,2) coordinate (b14);
\path (1.5,1.9) coordinate (b15);
\path (1.3,1.6) coordinate (b16);
\path (3.6,0.7) coordinate (b21);
\path (3.9,0.5) coordinate (b22);
\path (4,0.2) coordinate (b23);
\path (4,-0.2) coordinate (b24);
\path (3.9,-0.5) coordinate (b25);
\path (3.6,-0.7) coordinate (b26);

\foreach \i in {1,...,3}
{  \draw (a\i) -- (a); }
\foreach \i in {1,...,6}
{  \draw (a3\i) -- (a3); }
\foreach \i in {1,...,6}
{  \draw (a1\i) -- (a1); }
\foreach \i in {1,...,6}
{  \draw (a2\i) -- (a2); }
\foreach \i in {1,...,3}
{  \draw (b\i) -- (b); }
\foreach \i in {1,...,6}
{  \draw (b3\i) -- (b3); }
\foreach \i in {1,...,6}
{  \draw (b1\i) -- (b1); }
\foreach \i in {1,...,6}
{  \draw (b2\i) -- (b2); }

\draw (a) -- (b);
\draw (a1) -- (a2);
\draw (a1) -- (a3);
\draw (a3) -- (a2);
\draw (b1) -- (b2);
\draw (b1) -- (b3);
\draw (b3) -- (b2);

\foreach \i in {1,...,3}
{  \draw (a\i)  [fill=white] circle (\vr); }
\foreach \i in {1,...,3}
{  \draw (b\i)  [fill=white] circle (\vr); }
\foreach \i in {1,...,6}
\foreach \j in {1,...,3}
{  \draw (a\j\i)  [fill=white] circle (\vr); 
{  \draw (b\j\i)  [fill=white] circle (\vr); }
}
\draw (a)  [fill=white] circle (\vr);
\draw (b)  [fill=white] circle (\vr);
\draw (-1.4,1) node {$a_1$}; \draw (-2,0.3) node {$a_2$}; \draw (-1.4,-1) node {$a_3$};
\draw (2.4,1) node {$b_1$}; \draw (3,0.3) node {$b_2$}; \draw (2.4,-1) node {$b_3$};
\draw (0,0.3) node {$a$}; \draw (1,0.3) node {$b$};
\end{tikzpicture}
\end{center}
\caption{Graph $G_4$ with $\pch(G_4)=7$ and $\pch(G_4-ab)=4$}
\label{fig:Ex1}
\end{figure}

\begin{theorem}
\label{thm:realization}
For an arbitrary integer $k\ge 3$ and for an arbitrary integer $n$, where $\frac{k+1}{2} \leq n \leq k$, there exists a graph $G$ with an edge $e$ such that $\chi_\rho(G)=k$ and $\chi_\rho(G-e)=n$.
\end{theorem}

\begin{proof}
Let $k$ and $n$ be any integers such that $k\geq 3$ and $\frac{k+1}{2} \leq n \leq k$. 
For $n=k$, let the graph $G$ be obtained by attaching a leaf to one vertex of $K_k$. It is easy to see that $\chi_\rho(G)=k=\chi_\rho(G-e)$, where $e$ is an edge connecting a leaf and its support vertex.

If $n < k$, then construct the graph $G$ in a similar way as the graphs $G_n$ from the proof of Theorem~\ref{izrek1}. Namely, take one copy of $K_n$ and one copy of $K_{k+1-n}$, join them by a single edge $ab$, where $a \in V(K_n)$ and $b \in V(K_{k+1-n})$, and attach $k-1$ leaves to each vertex from $(V(K_n) \cup V(K_{k+1-n})) \setminus \{a,b\}$. Since $n \geq \frac{k+1}{2} \geq 2$ and $k+1-n \geq 2$, an analogous consideration as in the proof of Theorem~\ref{izrek1} yields that $\chi_\rho(G)=k$ and $\chi_\rho(G-ab)=n$. \qed
\end{proof}

The next lemma will be applied in Section~\ref{sec:block}
when we will study critical block graphs of small diameter. Nevertheless, it may be useful in a more general context when dealing with critical graphs for the packing chromatic number. 

\begin{lemma}
Let $e$ be an edge of $G$ such that $\diam(G-e) > \diam(G)=k$ and let $u, v$ be two vertices of $G$ for which $d_{G-e}(u,v)>k$. If there exists a $\chi_\rho(G)$-packing coloring $c$ of $G$ such that $c(v)>c(u) \geq k$, then $\chi_\rho(G-e) < \chi_\rho(G)$.
\label{lemma1}
\end{lemma}

\begin{proof}
Let $G$, $e\in E(G)$ and $u,v \in V(G)$ satisfy the conditions from the statement. 
Suppose that $c$ is a $\chi_\rho(G)$-packing coloring of $G$, which assigns to $u$ and $v$ two different colors, both greater than $k-1$. In order to prove that $\chi_\rho(G-e) < \chi_\rho(G)$ we construct a $(\chi_\rho(G)-1)$-packing coloring $c'$ of $G-e$. 

First, let $c'(u)=c'(v)=k$, which means that $|(c')^{-1}(k)| > |c^{-1}(k)|= 1$ (note that $\diam(G)=k$ implies $|c^{-1}(k)| \leq 1$). Then let $c'(a)=c(a)$ for all $a \in V(G) \setminus \{u,v\}$.
If there exists a vertex $x \in V(G) \setminus \{u,v\}$ for which $c'(x)=k$ (recall that $|c^{-1}(k)| \leq 1$, thus such vertex from $V(G) \setminus \{u,v\}$ is at most one), then substitute the color of $x$ with the color $c(u)$. Consequently the vertices of $G-e$ that are given color $k$ are pairwise at distance more than $k$. Also note that if $x$ was given color $c(u)$, then $x$ is the only vertex of $G-e$ with this color, since $c(u) \geq k$ and so $u$ was the only vertex with this color in $G$. 
Since the colors of all other vertices of $G$ (that is, vertices in $V(G) \setminus \{u,v,x\}$, respectively $V(G) \setminus \{u,v\}$) are unchanged, $c'$ is a packing coloring of $G-e$. In addition $c'$ uses less than $\chi_\rho(G)$ colors, since $|(c')^{-1}(i)| = |c^{-1}(i)| $ for all $i \in \{1,2, \ldots, k-1\}$, yet the color $c(v)$ is used by $c$ but not by $c'$. Thus, $\chi_\rho(G-e) < \chi_\rho(G)$. \qed
\end{proof}


\begin{lemma}
If $G$ is a $\chi_\rho$-critical graph, then $G$ is connected.
\label{claim1}
\end{lemma}

\begin{proof}
Suppose to the contrary that $G$ is not a connected graph and denote its connected components by $A_1, \ldots, A_r$, $r \geq 2$. Since $\chi_\rho(G) = \max_{1 \leq i \leq r}\{\chi_\rho(A_i)\}$, there exists $j \in \{1, 2, \ldots, r\}$ such that $\chi_\rho(A_j)=\chi_\rho(G)$. Let $x \in V(A_k)$, where $k \in \{1,2, \ldots, r\}\setminus\{j\}$.  Then $\chi_\rho(G-x)=\chi_\rho(A_j)=\chi_\rho(G)$, a contradiction to $G$ being critical. 
\qed
\end{proof}

We end this preliminary section with the following observation, which relates $\chi_\rho$-critical graph with $\chi_\rho$-vertex-critical graphs. 
\begin{lemma}
If $G$ is a $\chi_\rho$-critical graph, then it is also a $\chi_\rho$-vertex-critical graph. 
\label{claim2}
\end{lemma}
\begin{proof}
Since $G$ is $\chi_\rho$-critical, it is connected due to Lemma~\ref{claim1}. If $G$ is isomorphic to $K_1$, then it is clearly $\chi_\rho$-vertex-critical. Otherwise, let $x$ be an arbitrary vertex of $G$ and let $e=xy$ be an edge of $G$. Since $G-x$ is a subgraph of $G-e$ and $G$ is $\chi_\rho$-critical, it follows that $\chi_\rho(G-x) \leq \chi_\rho(G-e) < \chi_\rho(G)$. Thus, $G$ is $\chi_\rho$-vertex-critical. 
\qed
\end{proof}

\section{Critical graphs with small packing chromatic numbers and critical trees}
\label{sec:vertex}
We start this section by characterizing the $\pch$-critical graphs $G$ with small packing chromatic numbers, i.e, $\pch(G)\in\{2,3\}$. We will use the result of 
Goddard et al.~\cite[Proposition 3.1]{goddard-2008} which characterizes connected graphs $G$ with $\pch(G)=2$ as the star graphs ($G\cong K_{1,r}$, $r\ge 1$). 
We will also use the characterization of $\chi_\rho$-vertex-critical graphs $G$ with $\pch(G)=3$ by Klav\v zar and Rall~\cite[Proposition 4.1]{kr-2019} stating that these are precisely the graphs $G\in\{C_3,P_4,C_4\}$.

\begin{proposition}
\begin{enumerate}[(i)]
\item
A graph $G$ is $2$-$\chi_\rho$-critical if and only if $G \cong K_2$.
\item
A graph $G$ is $3$-$\chi_\rho$-critical if and only if $G \in \{C_3, P_4\}$.
\end{enumerate}
\end{proposition}
\begin{proof}
(i) Let $G$ be an arbitrary $2$-$\chi_\rho$-critical graph. By Lemma~\ref{claim1} $G$ is connected, which implies that $G$ is a star. Clearly, if $|V(G)|\geq 3$, then for each $e \in E(G)$, $G-e$ contains $K_2$ as induced subgraph, which implies that $\chi_\rho(G-e)\geq 2$, which is a contradiction. Therefore, $|V(G)|< 3$, and clearly $K_2$ is the only $2$-$\chi_\rho$-critical graph.
(ii) By Lemma~\ref{claim2}, each $\chi_\rho$-critical graph is also $\chi_\rho$-vertex-critical, which implies by~\cite{kr-2019} that $3$-$\chi_\rho$-critical graphs can only be $C_3$, $P_4$ or $C_4$. It is easy to see that $C_3$ and $P_4$ are $\pch$-critical, while $C_4$ is not, since $\pch(C_4)=\pch(C_4-e)=3$.
\qed
\end{proof}

While we could not find a general characterization of the packing chromatic critical graphs $G$ with $\pch(G)=4$, we give two partial results for these graphs. The first one considers graphs that contain a cycle $C_n$, where $n\ge 5$ is not divisible by $4$. Note that for such $n$, we have $\pch(C_n)=4$, while $\pch(C_n-e)=3$, thus the cycles themselves are $\pch$-critical. However, if $G$ has $\pch(G)=4$ and $G$ contains as a proper subgraph a cycle $C_n$, where $n\ge 5$ is not divisible by $4$, then $G$ is clearly not $\pch$-critical. We summarize this observation as follows.

\begin{proposition}
\label{prp:cikli}
If $G$ be a graph containing a cycle $C_n$, where $n\ge 5$ is an integer not divisible by $4$, then $G$ is a $4$-$\chi_\rho$-critical graph if and only if $G$ is isomorphic to $C_n$.
\end{proposition}

Let $\C$ be the class of graphs that contain exactly one cycle and have an arbitrary number of leaves  attached to each of the vertices of the cycle. The {\em net graph} is obtained by attaching a single leaf to each vertex of $C_3$. In the characterization of $\pch$-critical graphs $G$ with $\pch(G)=4$ within the graphs from the class $\C$, we use the result about $\pch$-vertex-critical graphs within the class $\C$. 

\begin{theorem}\cite{kr-2019}
\label{thm:kr}
A graph $G \in \C$ is $4$-$\chi_\rho$-vertex critical graph if and only if $G$ is one of the following graphs:
\begin{itemize}
\item $G \cong C_n$, $n \geq 5$, $n \not\equiv 0 \,\, (mod \, 4)$;
\item $G$ is the net graph;
\item $G$ is obtained by attaching a single leaf to two adjacent vertices of $C_4$.
\item $G$ is obtained by attaching a single leaf to two vertices at distance $3$ on $C_8$.
\end{itemize}
\end{theorem}

The following result shows that the first three instances of the above theorem work also in the more strict case of $\pch$-critical graphs. 
\begin{theorem}
A graph $G \in C$ is $4$-$\chi_\rho$-critical graph if and only if $G$ is one of the following graphs:
\begin{itemize}
\item $G \cong C_n$, $n \geq 5$, $n \not\equiv 0 \,\, (mod \, 4)$;
\item $G$ is the net graph;
\item $G$ is obtained by attaching a single leaf to two adjacent vertices of $C_4$.
\end{itemize}
\end{theorem}
\begin{proof}
Let $G \in C$ be an arbitrary $4$-$\chi_\rho$-critical graph. Since each $\chi_\rho$-critical graph is also $\chi_\rho$-vertex critical, Theorem~\ref{thm:kr} implies that only one of the four cases is possible for $G$. We now examine each of the cases. 

\textbf{Case 1.} $G \cong C_n$, $n \geq 5$, $n \not\equiv 0 \,\, (mod \, 4)$ \\
By Proposition~\ref{prp:cikli}, $G$ is $\pch$-critical. 

\textbf{Case 2.} $G$ is the net graph. \\
Let $a_1,a_2,a_3, b_1, b_2, b_3 \in V(G)$ such that $a_1$, $a_2$, $a_3$ are the leaves and $a_1b_1, a_2b_2, a_3b_3 \in E(G)$. If $e=a_ib_i$, $i \in \{1,2,3\}$, then by setting $c(a_i)=c(b_i)=c(a_j)=c(a_k)=1$, $c(b_j)=2$ and $c(b_k)=3$, $j, k \in \{1,2,3\}$, $j \neq i, k \neq i$, $j \neq k$, $c$ is a packing coloring of $G-e$ using $3$ colors. Otherwise, $e=b_ib_j$ for some $i,j \in \{1,2,3\}, i \neq j$. Then, by letting $c(a_1)=c(a_2)=c(a_3)=2$, $c(b_i)=c(b_j)=1$ and $c(b_k)=3$, $k \in \{1,2,3\}$, $k \neq i$, $k\neq j$, we again get a packing coloring $c$ of $G-e$ using $3$ colors. Therefore, the net graph is $4$-$\chi_\rho$-critical graph.

\textbf{Case 3.} $G$ is obtained by attaching a single leaf to two adjacent vertices of $C_4$. \\
Let $a, b, x, y, a_1, b_1 \in V(G)$ such that $\deg(a_1)=\deg(b_1)=1$, $aa_1, bb_1 \in E(G)$ and $ax, by \notin E(G)$. If $e=aa_1$, then by letting $c(a_1)=c(b_1)=c(a)=c(x)=1$, $c(b)=2$, $c(y)=3$, $c$ is a packing coloring of $G-e$ using $3$ colors (by symmetry, this also settles the case $e=bb_1$). If $e=ab$, then $G-e \cong P_6$ and thus $\chi_\rho(G-e)=3$. Next, suppose that $e=ay$. In this case, by letting $c(a)=c(b_1)=c(x)=1$, $c(a_1)=c(y)=2$ and $c(b)=3$, $c$ is a packing coloring of $G-e$ using $3$ colors (by symmetry, $e=bx$ is also settled). Finally, if $e=xy$, then coloring all leaves of $G-e$ with color $1$ and vertices $a$ and $b$ with different colors from $\{2,3\}$ yields a $3$-packing coloring of $G-e$, thus $G$ is a $4$-$\chi_\rho$-critical graph.

\textbf{Case 4.} $G$ is obtained by attaching a single leaf to two vertices at distance $3$ on $C_8$. \\
Denote by $a, b \in V(G)$ the vertices of degree $3$ and let $e=uv \in E(G)$ such that $d(u,a)=2$, $d(u,b)=3$, $d(v,b)=2$. It is easy to check (one can also use~ \cite[Proposition 3.3]{goddard-2008}) that $\chi_\rho(G-e)>3$, which implies that $G$ is not $\chi_\rho$-critical. 
\qed
\end{proof}


Next, we focus on $\pch$-critical trees.
By Lemma~\ref{claim2}, every $\chi_\rho$-critical graph is also $\chi_\rho$-vertex-critical. We prove that in trees these two types of critical graphs coincide.

\begin{theorem}
\label{thm:trees}
A tree $T$ is $\chi_\rho$-critical if and only if it is $\chi_\rho$-vertex-critical.
\end{theorem}
\begin{proof}
According to Lemma~\ref{claim2}, we only need to prove that a $\chi_\rho$-vertex-critical tree $T$ is $\pch$-critical. Note that for any edge $e$ of $T$, $T-e$ consists of two connected components, which we denote by $T_1$ and $T_2$. Since  $T_1$ is an induced subgraph of $T$, and $T$ is $\chi_\rho$-vertex critical, we infer $\chi_\rho(T_1) < \chi_\rho(T)$. Analogously we derive that $\chi_\rho(T_2) < \chi_\rho(T)$. Let $c_1$ be any optimal packing coloring of $T_1$ and $c_2$ any optimal packing coloring of $T_2$. Since there are no edges between $T_1$ and $T_2$, we can form a packing coloring $c$ of $G-e$ as follows: $c(u)=c_1(u)$ for all $u \in V(T_1)$ and $c(v)=c_2(v)$ for all $v \in V(T_2)$. Note that $c$ uses less than $\chi_\rho(T)$ colors and hence the statement holds.
\qed
\end{proof}

From Theorem~\ref{thm:trees} we directly derive two results about $\pch$-critical trees, which follow from~\cite[Proposition 5.1]{kr-2019} and~\cite[Theorem 5.2]{kr-2019}, dealing with $\pch$-vertex-critical trees.  

\begin{corollary}
For an arbitrary integer $k \geq 2$ there exists a $\chi_\rho$-critical tree $T$ with $\pch(T)=k$.
\end{corollary}

Recall that a {\em caterpillar} is a tree such that the removal of all its leaves results a path.

\begin{corollary}
A $\chi_\rho$-critical caterpillar $T$ exists if and only if $\pch(T) \leq  7$.
\end{corollary}

\section{$\chi_\rho$-critical graphs with diameter $2$}
\label{sec:diam2}
In this section, we prove a characterization of $\pch$-critical graphs with diameter 2. Maximum independent (alias stable) sets play an important role in these graphs. As usual, $\alpha(G)$ denotes the cardinality of a maximum independent set of a graph $G$. An independent set of $G$ of size $\alpha(G)$ is called $\alpha(G)$-set, or shortly, $\alpha$-set, when no confusion can arise.

Recall that $\chi_\rho(G) \leq |V(G)|-\alpha(G)+1$ holds for any graph $G$~\cite{goddard-2008}. Moreover, if $G$ has diameter $2$, then $\chi_\rho(G) = |V(G)|-\alpha(G)+1$ \cite{goddard-2008}. These two facts will be used several times in this section.

\begin{theorem}
\label{thm:diam2}
If $G$ is a graph with diameter $2$, then $G$ is $\chi_\rho$-critical if and only if for each edge $e=u_1u_2 \in E(G)$  at least one of the following statements holds:
\begin{enumerate}[(i)]
\item $\alpha(G-e)>\alpha(G)$;
\item there exist a vertex $y \in N[u_i]$ such that $d_{G-e}(y,u_{j}) \geq 3$, where $\{i,j\}=\{1,2\}$, and an $\alpha(G)$-set $A$ such that $A \cap \{y, u_{j}\}= \emptyset$.
\end{enumerate}
\label{theorem_diam2}
\end{theorem}
\begin{proof}
Let $G$ be a $\chi_\rho$-critical graph with diameter $2$ and let $e=u_1u_2$ be an arbitrary edge of $G$. 
Since $G$ is $\pch$-critical, we have two possibilities for an optimal packing coloring $c$ of $G-e$. Notably, $c$ assigns color $1$ to at least $\alpha(G)+1$ vertices, which implies $\alpha(G-e) > \alpha(G)$, or $c$ assigns color $1$ to $k \leq \alpha(G)$ vertices and the remaining vertices are assigned colors from $\{2, 3, \ldots, \chi_{\rho}(G)-1 \}$ in such a way that at least two of these vertices receive the same color.  

First, we prove that $\diam(G-e) \leq 4$. If $\diam(G-e) \geq 5$, then there exist $a,b \in V(G)$ such that $d_{G-e}(a,b) \geq 5$ (and $d_G(a,b) \leq 2$). This implies that the endvertices of $e$ are $a$ and $b$, $a$ and $x\in N(b)$ (note that $x$ is not necessarily on the shortest $a,b$-path), or $b$ and $x\in N(a)$. In each of these cases one can derive that there exist two vertices lying on a shortest $a,b$-path in $G-e$, which are at distance at least $3$ in $G$; a contradiction to $\diam(G)=2$. Thus, $\diam(G-e) \leq 4$.

If $\diam(G-e)=2$, then any optimal packing coloring of $G-e$ assigns each of the colors from $\{2, 3, \ldots, \chi_\rho(G-e)\}$ to exactly one vertex. The fact that $G$ is $\chi_\rho$-critical implies $\alpha(G-e) > \alpha(G)$, thus in this case we are done since (i) holds.

Next, consider the case when $\diam(G-e) \geq 3$. 
If $\alpha(G-e) > \alpha(G)$, we are done, hence suppose that $\alpha(G-e) \leq \alpha(G)$. Let $c$ be any optimal packing coloring of $G-e$. We claim that there does not exist three vertices of $G-e$, which receive the same color $i \geq 2$ by $c$. 
Suppose to the contrary that there exist $x_1, x_2, x_3 \in V(G)$ such that $c(x_1)=c(x_2)=c(x_3)=i \geq 2$. Then, $d_{G-e}(x_j,x_k) \geq 3$ for each $j, k \in \{1,2,3\}$, $j \neq k$. Since $\diam(G)=2$, $e$ lies on each of the shortest $x_jx_k$-paths, $j, k \in \{1,2,3\}$, $j \neq k$. But this is not possible, since it implies that at least two vertices from $\{x_1,x_2,x_3\}$ must be at distance at most $2$ in $G-e$. Thus, the claim is true, that is, $c$ assigns each color from $\{2, 3, \ldots, \chi_\rho(G-e)\}$ to at most two vertices. Recall that $\diam(G-e) \leq 4$, which implies that $|c^{-1}(i)|=1$ for each $i \geq 4$. Therefore, since $G$ is $\chi_\rho$-critical we have: $|c^{-1}(2)|=2$ and $|c^{-1}(3)|=1$  (respectively, $|c^{-1}(3)|=2$ and $|c^{-1}(2)|=1$), or $|c^{-1}(2)|=2$ and $|c^{-1}(3)|=2$. 

First, suppose that $|c^{-1}(j)|=2$ for some $j \in \{2,3\}$ and $|c^{-1}(i)|=1$ for each $i \in \{2,3, \ldots, \chi_\rho(G-e)\} \setminus \{j\}$. Let $u,v \in V(G)$ such that $c(u)=c(v)=j$, $j \in \{2,3\}$. Then $d_{G-e}(u,v) \geq j+1 \geq 3$ and since $\diam(G)=2$, the endvertices of $e=u_1u_2$ are the vertices $u$ and $v$ ($u_1=u$ and $u_2=v$), $u$ and $x \in N(v)$ ($u_1=u$ and $u_2=x$, $v \in N(u_2)$) or $v$ and $x \in N(u)$ ($u_2=v$ and $u_1=x$, $u \in N(u_1)$). 
In other words, $\{u,v\}=\{y,u_{3-k}\}$, where $y \in N[u_k]$ for some $k \in \{1,2\}$.
Since $|c^{-1}(i)|=1$ for each $i \in \{2,3, \ldots, \chi_\rho(G-e)\} \setminus \{j\}$ and $G$ is $\chi_\rho$-critical, there exists an ${\alpha}$-set $A$ in $G-e$  with the property that $A \cap \{y, u_{3-k}\}= \emptyset$. Using the fact that $A$ does not contain at least one endvertex of $e$, we infer that  $A$ is also an $\alpha$-set of $G$, which completes the proof in this case.

Next, suppose that $|c^{-1}(2)|=2$, $|c^{-1}(3)|=2$ and $|c^{-1}(i)|=1$ for each $i \in \{4, 5, \ldots, \chi_\rho(G-e)\}$. Let $u,v \in V(G)$ such that $c(u)=c(v)=2$ and let  $z, w \in V(G) \setminus \{u, v\}$ such that $c(z)=c(w)=3$. Since $d_{G-e}(u,v) \geq 3$, $d_{G-e}(z,w) \geq 4$ and $\diam(G)=2$, each of the shortest $u,v$-paths and each of the shortest $z,w$-paths in $G$ contain the edge $e$. This implies that for some $k \in \{1,2\}$ we have: $\{u,v\}=\{y,u_{3-k}\}$, $y \in N(u_k)$, $\{z,w\}=\{y', u_k\}$ and $y' \in N(u_{3-k})$. Since $\diam(G)=2$, there exists $x \in N(y) \cap N(y')$. Since $y \in N(u_k)$, we get $d_{G-e}(z,w)\leq 3$, because $P:y',x,y,u_k$ is a $z,w$-path in $G-e$ of length 3. With this contradiction the first implication of the theorem is proven. 

For the converse, suppose that $G$ has diameter $2$ and for each edge of $G$ at least one of the statements (i) and (ii) holds. Let $e$ be an arbitrary edge of $G$. 
If $\alpha(G-e) \geq \alpha(G)+1$, then $\chi_\rho(G-e) \leq |V(G)| - \alpha(G) -1 + 1 < \chi_\rho(G)$, and we are done. 
Otherwise, for some $k \in \{1,2\}$ there exists an $\alpha$-set $A$ of $G$, which does not contain the vertices $u_{3-k}$ and $y \in N[u_k]$ with the property $d_{G-e}(u_{3-k},y) \geq 3$. In this case, a coloring $c$ of $G-e$, which assigns color $2$ to the vertices $u_{3-k}$ and $y$, color $1$ to all vertices from $A$, and $|V(G)|-\alpha(G)-2$ distinct colors to other vertices, is a packing coloring. Note that $c$ uses $|V(G)|-\alpha(G)$ colors, which is less than $\pch(G)= |V(G)|-\alpha(G)+1$. \qed
\end{proof}

\section{Critical block graphs}
\label{sec:block}
In this section, we study $\chi_\rho$-critical block graphs. Recall that a {\em block} of a graph $G$ is a maximal connected subgraph of $G$, which has no cut vertices (that is, a maximal 2-connected subgraph or a $K_2$ whose edge is a cut-edge of $G$). A graph in which each block is a complete graph, is called a {\em block graph}. 
Using Theorem~\ref{theorem_diam2}, it is not difficult to characterize $\chi_\rho$-critical block graphs with diameter $2$ (as we will see later, the task gets much harder for $\chi_\rho$-critical block graphs with diameter $3$).

\begin{theorem}
\label{thm:diam2block}
If $G$ is a block graph with diameter $2$, then $G$ is $\chi_\rho$-critical if and only if  $\delta(G) \geq 2$.
\end{theorem}

\begin{proof}
Let $G$ be a block graph with diameter $2$ and let $G$ consist of $k$ blocks. Note that the center of $G$ contains just one vertex $x$, which is adjacent to all other vertices of $G$. Note that $\alpha(G)=k$ and each $\alpha$-set of $G$ contains exactly one vertex (different from $x$) of each block of $G$. 

Let $G$ be $\chi_\rho$-critical. We claim that $\delta(G) \geq 2$. Suppose to the contrary that there exists $u \in V(G)$ such that deg$(u)=1$. Note that $u$ is contained in each $\alpha(G)$-set. Hence $e=ux$ does not satisfy the properties from Theorem \ref{theorem_diam2} and so $G$ is not $\chi_\rho$-critical, a contradiction to our assumption.

Conversely, suppose that $\delta(G) \geq 2$. Let $e=uv$ be an arbitrary edge of $G$ and denote by $B$ the block of $G$ such that $e \in E(B)$. If $u \neq x$ and $v\neq x$, then $\alpha(G - e) > \alpha(G)$, since $\alpha(G-e)$-sets contain both $u$ and $v$ and $k-1$ vertices from blocks different from $B$.  On the other hand, let one endvertex of $e$ be $x$, say $u=x$. Then, $d_{G-e}(v, y)= 3$, for some $y \in V(G) \setminus V(B)$. Note that $y \in N(u)$. Since $\delta(G) \geq 2$, each block of $G$ is of order at least $3$ and thus there exists an $\alpha(G)$-set $A$ such that $A \cap \{v, y\}= \emptyset$. Therefore, Theorem \ref{theorem_diam2} implies that $G$ is $\chi_\rho$-critical.
\qed
\end{proof}

Now, let $G$ be a block graph with diameter $3$. Note that then $\rad(G)=2$, and the center of $G$ consists of (at least two) vertices that form a block. We call this block the {\em central block} of $G$, and all other blocks of $G$ will be called {\em side blocks}. Since $\diam(G)=3$, $G$ has at least two side blocks, which do not intersect, yet they intersect with two distinct vertices of the central block.  

We follow with one of the main results of this paper---a characterization of $\pch$-critical block graphs with diameter 3. 

\begin{theorem}
Let $G$ be a block graph with diameter $3$, and let $B$ be the central block of $G$. Graph $G$ is $\chi_\rho$-critical if and only if one of the following three possibilities holds for the vertices of $B$.\\
\textbf{$(a)$} All vertices in $V(B)$ have degree $|V(B)|$.\\
\textbf{$(b)$} All vertices in $V(B)$ have degree $|V(B)|+1$, and exactly $|V(B)|-1$ vertices of $B$ have two leaf neighbors. \\ 
\textbf{$(c)$} For each vertex $x \in V(B)$ at least one of the following three properties holds:

\textbf{$(c1)$} $x$ belongs to at least one side block of order at least $4$, but does not have any leaf neighbor; 

\textbf{$(c2)$} $x$ belongs to at least two side blocks of order $3$, but does not have any leaf neighbor;

\textbf{$(c3)$} $x$ has degree $|V(B)|+1$ and has two neighbors, which are both leaves;\\
in addition, at least one vertex in $V(B)$ satisfies one of the properties $(c1)$ or $(c2)$.
\end{theorem}

\begin{proof}
For each vertex $x$ of the central block $B$, let $B_x$ denote the subgraph induced by $N(x) \setminus V(B)$. 

To prove the first implication, let $G$ be a $\chi_\rho$-critical graph.  The  assumption that the vertices of $B$ do not satisfy the properties (a), (b) and (c) will lead us to a contradiction. We distinguish two cases, which are then further divided upon different possibilities.

\textbf{Case 1.} Suppose that there exists at least one vertex in $V(B)$, which satisfies one of the properties $(c1)$ or $(c2)$. Denote such vertex by $x$. Then there exists $y \in V(B)$, which does not satisfy the properties $(c1)$, $(c2)$, $(c3)$. 
We distinguish the following four cases with respect to the number of leaf neighbors of $y$. 

\textbf{Case 1.1} Vertex $y$ is adjacent to at least three leaves. \\
Let $e=yu$, where $u \in V(B_y)$ is a leaf. Now, we prove that there 
exists an optimal packing coloring $c$ of $G-e$, such that $c(y)>1$. Suppose that $c'$ is a $\pch(G-e)$-packing coloring with $c'(y)=1$. Then the leaves adjacent to $y$ receive different colors by $c'$, which are greater than $1$. In particular, there exists a leaf $v$ such that $c'(v)\ge 3$, and so the color $c'(v)$ appears only on $v$. By exchanging the colors of $v$ and $y$ (letting $c(v)=1$ and $c(y)=c'(v)$) and setting $c(u)=1$, we get a packing coloring $c$ of $G-e$ using $\pch(G-e)$ colors. We infer that $c$ is also a packing coloring of $G$, hence $\pch(G)\le \pch(G-e)$, which contradicts that $G$ is $\chi_\rho$-critical.

\textbf{Case 1.2} Vertex $y$ is adjacent to exactly two leaves. \\
Since $y$ does not satisfy the property $(c3)$, it is contained in at least one side block of order at least $3$. In this case, an analogous consideration as is Case 1.1 yields the result.

\textbf{Case 1.3} Vertex $y$ is adjacent to exactly one leaf. \\
Let $e=yu$, where $u \in V(B_y)$ is a leaf. If $y$ is contained in a side block of order at least $3$, then the proof is analogous as in Case 1.1. 
Otherwise, $V(B_y)=\{u\}$. Suppose that $c'$ is a $\pch(G-e)$-packing coloring with $c'(y)\ge 2$. Then, by letting $c(u)=1$, and $c(z)=c'(z)$ for all $z\in V(G)\setminus\{u\}$, $c$ is a packing coloring of $G$ using the same number of colors as $c'$. Hence, $\pch(G)\le \pch(G-e)$, which is a contradiction. Now, let $c'(y)=1$. Suppose that $c'(v)=2$, where $v\in V(B)$. Then the color $2$ appears only on $v$ with respect to $c'$ (and perhaps on $u$). In addition, there exists a vertex $a\in V(B_x)$ such that $c'(a)\ge 3$, hence the color $c'(a)$ appears only on $a$ and perhaps on $u$ with respect to $c'$. By replacing the colors of vertices $v$ and $a$, notably, letting $c(a)=2$, $c(v)=c'(a)$, and setting $c(u)=2$, $c(z)=c'(z)$ for all $z\in V(G)\setminus\{a,v,u\}$, we infer that $c$ is a packing coloring of $G$ using at most $\pch(G-e)$ colors. So we again get $\chi_\rho(G) \leq \chi_\rho(G-e)$, a contradiction to $G$ being critical. Finally, if no vertex of $B$ receives color $2$ by $c'$,  then already $c'$ is a packing coloring of $G$, setting $c'(u)=2$ if necessary.

\textbf{Case 1.4} Vertex $y$ is not adjacent to any leaf. \\
This implies, that $y$ is either contained in a side block of order $3$, or it is a simplicial vertex. 

First, consider the case when $y$ is contained in $B$ and in one block isomorphic to $K_3$.
Let $e=uv$, where $u,v \in V(B_y)$. 
Suppose that $c'$ is a $\chi_\rho(G-e)$-packing coloring of $G-e$ with $c'(y)=1$. This implies that $c'(u) \neq c'(v)$. Since $d_G(v_1,v_2)=d_{G-e}(v_1,v_2)$ for all $v_1,v_2 \in V(G)$ except for $\{v_1,v_2\}=\{u,v\}$, we infer that $c'$ is also a packing coloring of $G$
using $\chi_\rho(G-e)$ colors. Thus, $\chi_\rho(G) \leq \chi_\rho(G-e)$, a contradiction to $G$ being critical. 
Now, let $c'(y) \neq 1$. Suppose that $c'(w) \neq 2$, for all $w \in V(B)$. Then, by letting $c(u) = 1$, $c(v)=2$ and $c(z) = c'(z)$ for all $z \in V(G) \setminus \{u,v\}$, we infer that $c$ is a packing coloring of $G$ using the same number of colors as $c'$. Hence, $\chi_\rho(G) \leq \chi_\rho(G-e)$, which is a contradiction. Then, suppose that $c'(w) = 2$ for some $w \in V(B)$. This implies that color $2$ appears only on $w$ with respect to $c'$. In addition, there exists a vertex $a \in V(B_x)$ such that $c'(a) \geq 3$, hence the color $c'(a)$ appears only on $a$ with respect to $c'$. By replacing the colors of vertices $w$ and $a$, notably, letting $c(a) = 2$, $c(w) = c'(a)$ and setting $c(u) = 1$, $c(v)=2$ and $c(z)=c'(z)$ for all $z \in V(G) \setminus \{a,w,u,v \}$, we infer that $c$ is a packing coloring of $G$
using at most $\chi_\rho(G-e)$ colors, a contradiction to $G$ being critical. 

It remains to consider the case when $y$ is a simplicial vertex. 
Since $\diam(G)=3$, there exists $v \in V(B) \setminus \{x\}$ which is not simplicial. If $v$ does not satisfy the properties $(c1)$, $(c2)$ and $(c3)$, then by switching the roles of $y$ and $v$, we infer the result analogously as in Cases $1.1$, $1.2$, $1.3$, or $1.4$. 
Otherwise, let $e=yv$ and let $c'$ be a $\chi_\rho(G-e)$-packing coloring of $G-e$. Note that there exists a vertex $a \in V(B_x)$ such that $c'(a) \geq 3$ and hence the color $c'(a)$ appears only on $a$ with respect to $c'$. If $c'(y) \neq 1$, then by exchanging the colors of $y$ and $a$ (letting $c(y) = c'(a)$ and $c(a) = c'(y)$) and $c(z)=c'(z)$ for all $z \in V(G) \setminus \{y,a\}$, we get a packing coloring $c$ of $G$ using $\chi_\rho(G - e)$ colors, a contradiction to $G$ being critical.
Therefore, $c'(y)=1$. If $c'(v) \neq 1$, then we infer that $c'$ is also a packing coloring of $G$, hence $\chi_\rho(G) \leq \chi_\rho(G-e)$, which contradicts that $G$ is $\chi_\rho$-critical. Therefore, $c'(y)=c'(v)=1$ and since $v$ satisfies the property $(c1)$, $(c2)$ or $(c3)$, there exists $v' \in V(B_v)$ such that $c'(v') \geq 3$. This implies that color $c'(v')$ appears only on $v'$ with respect to $c'$. By replacing the colors of vertices $v$ and $v'$, notably, letting $c(v) = c'(v')$, $c(v') = 1$ and $c(z)=c'(z)$ for all $z \in V(G) \setminus \{v, v' \}$, we infer that $c$ is a packing coloring of $G$ using at most $\chi_\rho(G-e)$ colors, a contradiction to $G$ being critical.

\textbf{Case 2} None of the vertices from $V(B)$ satisfies $(c1)$ or $(c2)$. 

\textbf{Case 2.1} $G$ contains at least one side block of order at least $3$. \\ 
Let $x \in V(B) \cap V(C)$, where $C$ is a side block of $G$, which is not isomorphic to $K_2$. Since $\diam(G)=3$, there exists $y \in V(B)\setminus\{x\}$ such that deg$(y) \geq |V(B)|$. 
If $y$ is adjacent to at least three leaves or to exactly one leaf, then the proof is analogous as in Case 1.1. respectively Case 1.3. If $y$ is not a support vertex, it is contained in exactly two different blocks of $G$, one of which is $B$ and the other is isomorphic to $K_3$ (since $y$ does not satisfy $(c1)$ or $(c2)$). In this case the proof is analogous as in Case 1.4. 
Therefore, each non-simplicial vertex (except perhaps $x$) of $B$ is adjacent to exactly two leaves. 
Next, suppose that there exists a simplicial vertex $v \in V(B)$. Let $e = yv$ and let $c'$ be a $\pch(G-e)$-packing coloring of $G-e$. If $c'(v)=1$ and $c'(y) \neq 1$, then $c'$ is also a packing coloring of $G$ using $\chi_\rho(G-e)$ colors, which contradicts that $G$ is $\chi_\rho$-critical. If $c'(v)=c'(y)=1$, then there exists a leaf $y' \in V(B_y)$ such that $c'(y') \geq 3$ and therefore $c'(y')$ appears only on $y'$ with respect to $c'$. By exchanging the colors of $y$ and $y'$ (letting $c(y) = c'(y')$ and $c(y')=1$) and $c(z) = c'(z)$ for all $z\in V (G) \setminus \{y,y'\}$, we get a packing coloring $c$ of $G$ using $\pch(G-e)$ colors, what is again a contradiction to G being critical. Therefore, $c'(v) \neq 1$. If $c'(v) \geq 3$, then $c'(v)$ appears only on $v$ with respect to $c'$, thus $c'$ is also a packing coloring of $G$. This yields a contradiction, since $\chi_\rho(G) \leq \pch(G-e)$. Hence, suppose $c'(v)=2$ and note that $c'(v') \neq 2$ for each $v' \in V(B) \setminus \{v\}$. Note that there exists $x' \in V(B_x)$ such that $c'(x') \geq 3$ and this color appears only on $x'$. By replacing the colors of vertices $v$ and $x'$, notably, letting $c(v) = c'(x')$, $c(x') = 2$ and $c(z) = c'(z)$ for all $z \in V (G) \setminus \{v,x'\}$, we infer that $c$ is a packing coloring of $G$ using at most $\pch(G-e)$ colors, a contradiction. Therefore, $B$ does not contain simplicial vertices. 

Now, suppose that there exists $a \in V(B)$, such that $a$ is contained in at least one side block isomorphic to $K_2$ and in at least one side block isomorphic to $K_3$. Let $e = ab$, where $b$ is a leaf, and let $c'$ be a $\pch(G-e)$-packing coloring of $G-e$. If $c'(a)=1$, then there exist $a' \in V(B_a)$ such that $c'(a') \geq 3$. By exchanging the colors of $a$ and $a'$ (letting $c(a') = 1$ and $c(a) = c'(a)$) and setting $c(b) = 1$, we get a packing coloring $c$ of $G-e$ using $\chi_\rho(G-e)$ colors. We infer that $c$ is also a packing coloring of $G$, thus $\pch(G) \leq \pch(G-e)$, which contradicts that $G$ is $\chi_\rho$-critical. Finally, if $c'(a) \neq 1$, then already $c'$ is a packing coloring of $G$, setting $c'(b) = 1$ if necessary. We have $\chi_\rho(G) \leq \chi_\rho(G-e)$, a contradiction. 

We conclude that all vertices of $B$ are of degree $|V(B)|+1$. Since they do not satisfy the property $(b)$, there exist $x,x' \in V(B)$, which are not support vertices. Let $u,v \in V(B_{x'})$ and let $e=uv$. The proof that $\chi_\rho(G) \leq \chi_\rho(G-e)$ is analogous as in Case 1.4, which is the last contradiction of this case.

\textbf{Case 2.2} All side blocks of $G$ are isomorphic to $K_2$. 

First, suppose that there exists a simplicial vertex $a\in V(B)$. Since $\diam(G)=3$, there also exist $u \in V(B)$ and $v \in V(B_u)$. Let $e=uv$. 
Note that $\chi_\rho(G-e) \geq |V(B)|$ since $B$ is a complete subgraph of $G-e$. Then, by setting $c(a)=1$, $c(v)=1$, $c(u)=2$, $c(l)=1$ for each leaf $l$ of $G-e$, and $c(v_i)=i$ for each $v_i \in V(B) \setminus \{a,u\}$, $i \in \{3, 4, \ldots, |V(B)|\}$, we get a packing coloring $c$ of $G$ using $|V(B)|$ colors. Thus, $\chi_\rho(G) \leq \chi_\rho(G-e)$, a contradiction to $G$ being critical. 

Therefore, there does not exist $a\in V(B)$, which is simplicial in $G$. 
Since the vertices from $V(B)$ do not satisfy the property $(a)$, there exists $u \in V(B)$ such that deg$(u)\geq|V(B)|+1$. Let $v,v' \in V(B_u)$ and let $e=uv$. 
Then, $G-e$ contains a subgraph $H$ such that $V(B) \subset V(H)$ and each vertex of $H$ is of degree $|V(B)|$ or $1$. Let $c$ be an optimal packing coloring of $H$. First, we prove that $\chi_\rho(H) \geq |V(B)|+1$. Since $B$ is a complete graph, all
vertices of $B$ receive different colors by a packing coloring $c$. If no vertex of $B$ receives color $1$, it is clear that $\chi_\rho(H) \geq |V(B)|+1$. On the other hand, if $c(p)=1$ for $p\in V(B)$, then for the leaf neighbor ${q}$ of $p$ we get $c(q)>1$. Hence the color $c({q})$ does not appear among vertices of $B$, which again implies that $\chi_\rho(H) \geq |V(B)|+1$.
Since $H$ is a subgraph of $G-e$ it follows that $\chi_\rho(G-e) \geq |V(B)|+1$. By setting $c(u)=2$, $c(v)=c(v')=1$, $c(l)=1$ for all leaves $l \in V(G)$, and $c(v_{i-1})=i$ for all $v_{i-1} \in V(B) \setminus \{u\}$, $3 \leq i \leq |V(B)|+1$, we infer that $c$ is a packing coloring of $G$ using at most $|V(B)|+1$ colors. So we again get $\chi_\rho(G) \leq \chi_\rho(G-e)$, which contradicts that $G$ is $\chi_\rho$-critical. By this, one direction of the proof is complete.


For the converse implication, let us assume that vertices in $V(B)$ satisfy one the properties $(a)$, $(b)$ or $(c)$, and prove that $G$ is $\chi_\rho$-critical.  That is, for each edge $e \in E(G)$ we will show that $\chi_\rho(G-e) < \chi_\rho(G)$. Denote by $c$ an optimal packing coloring of $G$ and let $e=uv$ be an edge of $G$ that will be removed. 

\textbf{Case a.} The vertices from $V(B)$ satisfy the property $(a)$. \\
One can prove that $\chi_\rho(G) \geq |V(B)|+1$ along the same lines as proving that $\chi_\rho(H) \geq |V(B)|+1$ in Case 2.2 (in fact, the graphs are isomorphic).
Next, we prove that $\chi_\rho(G-e) \leq |V(B)|$. If $u, v \in V(B)$, then form a packing coloring $c'$ of $G-e$ using $|V(B)|$ colors as follows. Let $c'(u)=c'(v)=1$, $c'(l)=2$ for all leaves $l$ of $G$, and let the other vertices of $G$ receive different colors from $\{3, \ldots, |V(B)|\}$ by $c'$. Otherwise, $u \in V(B)$ and $v \notin V(B)$. In this case, let $c'(u)=c'(v)=1$, $c'(l)=1$ for all leaves $l$ of $G$, and let the vertices from $V(B) \setminus \{u\}$ receive different colors from $\{2, \ldots, |V(B)|\}$ by $c'$. Since $c'$ uses $|V(B)|$ colors, we infer that $\chi_\rho(G-e) \leq |V(B)|$, which implies that $G$ is $\chi_\rho$-critical.

\textbf{Case b.} The vertices from $V(B)$ satisfy the property $(b)$. \\
Denote by $t$ the vertex from $V(B)$, which is not a support vertex.

First, prove that $\chi_\rho(G)\geq|V(B)|+2$. Note that if $c(a)=1$ for $a\in V(B)$, the vertices $a',a''\in V(B_a)$ receive colors greater than $1$. The colors $c(a'),c(a'')$ thus cannot appear on vertices of $B$. This readily implies that $c$ uses at least $|V(B)|+2$. Hence $c(a)>1$ for all $a\in V(B)$, and we may assume that the vertices in $B$ receive colors from $\{2,\ldots,|V(B)|+1\}$. Note that one of the vertices $t'$ in $V(B_t)$ has $c(t')>1$. Since $t'$ is at distance at most $2$ from vertices in $B$, we find that $c(t')>|V(B)|+1$. Thus $\chi_\rho(G)\geq|V(B)|+2$.


In order to prove that $\chi_\rho(G-e) \leq |V(B)|+1$ for all $e=uv$  , first suppose that $u, v \in V(B)$. Form a packing coloring $c'$ of $G-e$ using $|V(B)|+1$ colors as follows. Let $c'(u)=c'(v)=1$, $c'(u')=2, c'(u'')=3, c'(v')=2, c'(v'')=3$ for $u', u'' \in V(B_u)$, $v', v'' \in V(B_v)$. Further, let the vertices from $V(G) \setminus (V(B) \cup V(B_u) \cup V(B_v))$ receive colors from $\{1,2\}$ by $c'$ and let the other vertices of $G$ receive different colors from $\{4, \ldots, |V(B)|+1\}$ by $c'$. Therefore, $\chi_\rho(G-e) \leq |V(B)|+1$. 
If $u \in V(B)$ and $v \notin V(B)$, then form a packing coloring $c'$ of $G-e$ as follows. If $u \in V(B) \setminus \{t\}$, then let $c'(u)=c'(v)=1$, $c'(t')=1$ for $t' \in V(B_t)$, $c'(l)=1$ for all leaves from $V(G) \setminus V(B_u)$, and $c'(u')=c(t'')=2$ for $u' \in V(B_u)$, $t'' \in V(B_t)$. Otherwise, $u=t$ and let $c'(u)=c'(v)=1$, $c'(l)=1$ for all leaves $l \in V(G)$ and $c'(u')=2$ for $u' \in V(B_u)$. In both cases $|(c')^{-1}(1) \cup (c')^{-1}(2)|= 2|V(B)|+1$ and thus $\chi_\rho(G-e) \leq |V(G-e)|-(2|V(B)|+1)+2 = |V(B)|+1$. 
Finally, suppose that $u,v \notin V(B)$ and by letting $c'(l)=1$ for each leaf $l$ of $G-e$ and $c'(a)=2$ for some $a \in V(B)$, $c'$ is a packing coloring of $G-e$ using at most $|V(B)|+1$ colors, which implies that $\chi_\rho(G-e) < \chi_\rho(G)$.


\textbf{Case c.} The vertices from $V(B)$ satisfy the property $(c)$. \\
Denote by $p_3$ the number of vertices from $V(B)$ with the property $(c3)$ and by $k$ the number of blocks of $G$. First, we prove that $|c^{-1}(1) \cup c^{-1}(2)|\leq k-1+|V(B)|-p_3$ (which implies $\chi_\rho(G) \geq |V(G)|-k+1-|V(B)|+p_3+2$). 

Suppose that $c(a)=2$ for some $a \in V(B)$. Then $c(x) \neq 2$ for all $x \in V(G) \setminus \{a\}$, since $d(a,x) \leq 2$. 
If $c(b)=1$ for some $b \in V(B)$, then $c(b') \neq 1$ for all $b' \in V(B_b)$, which implies that $|c^{-1}(1)| \leq k-1$.
Otherwise, at most most one vertex in each side block receives color $1$ by $c$, which again implies $|c^{-1}(1)| \leq k-1$. Therefore, in both cases we have $|c^{-1}(1) \cup c^{-1}(2)|\leq k \leq k-1+|V(B)|-p_3$, since $B$ contains at least one vertex which does not satisfy the property $(c3)$. 

Next, suppose that $c(a) \neq 2$ for all $a \in V(B)$. Given a vertex $b \in V(B)$ with the property $(c1)$ or $(c2)$, let $B_1^b, \ldots, B_{l_b}^b$ be the side blocks of $G$ that contain $b$. Since each $B_i^b$ is of order at least $3$, at most one vertex in $B_i^b$ receives color $1$ by $c$ and at most one vertex from $V(B_1^b) \cup \cdots \cup V(B_{l_b}^b)$ receives color $2$ by $c$. Therefore, the overall number of vertices in all side blocks that contain a vertex $b \in V(B)$ with the property $(c1)$ or $(c2)$, which receive color $1$ or $2$ by $c$, is most $k-(2p_3+1)+|V(B)|-p_3$. Now, consider vertices $b \in V(B)$ with the property $(c3)$ (that is, $b$ is a support vertex). If $c(b)=1$, then there exists $b' \in V(B_b)$ such that $c(b') \notin \{1,2\}$. Therefore at most $2p_3-1$ leaves receive color in $\{1, 2\}$ by $c$. Thus, $|c^{-1}(1) \cup c^{-1}(2)|\leq k-(2p_3+1)+|V(B)|-p_3+2p_3-1+1=k-1+|V(B)|-p_3$. Otherwise, $c(b) \neq 1$ for all support vertices $b \in V(B)$, and at most $2p_3$ leaves receive a color from $\{1, 2\}$ by $c$, which implies the same bound,$|c^{-1}(1) \cup c^{-1}(2)|\leq k-1+|V(B)|-p_3$ . 

Next, construct (an optimal) packing coloring $c_0$ of $G$ as follows. For each $b \in V(B)$ that does not satisfy the property $(c3)$, let any vertex from each set $V(B_i^b)\setminus \{b\}$, $1 \leq i \leq l_b$, receive color $1$. In addition, let any (other) vertex from $(V(B_1) \cup \ldots \cup V(B_{l_b}^b))\setminus \{b\}$ receive color $2$ by $c_0$. Further, let $\{c_0(b'),c_0(b'')\}=\{1,2\}$ for $b',b'' \in V(B_b)$, where $b \in V(B)$ satisfies the property $(c3)$. Finally, let all the remaining vertices of $G$ receive pairwise distinct colors by $c_0$. Since $|c_0^{-1}(1)|=k-1-p_3$ and $|c_0^{-1}(2)|=|V(B)|$, the coloring $c_0$ is an optimal packing coloring of $G$. Additionally, note that $c_0(b) \neq 1$ for all $b \in V(B)$.

In order to prove, that $\chi_\rho(G-e) < \chi_\rho(G)$, we will denote by $c'$ a packing coloring of $G-e$ and consider three subcases with respect to the types of edges $e=uv$ that are removed.

\textbf{Case c.1} $u,v \in V(B)$. \\
Since $d_{G-e}(u',v') = 4$ for $u' \in V(B_u)$ and $v' \in V(B_v)$, we have $\diam(G-e)>\diam(G)$. If $u$ and $v$ do not satisfy the property $(c3)$, then there exist $u' \in V(B_u)$ and $v' \in V(B_v)$, such that $c_0(u')=i \geq 3$ and $c_0(v')=j \geq 3$ (and $d(u',v')=4$, $i \neq j$). Therefore, using Lemma \ref{lemma1}, we infer that $\chi_\rho(G-e) < \chi_\rho(G)$. 
Next, suppose that only $u$ satisfies one of the properties $(c1)$ or $(c2)$. Recall that $c_0(v) \geq 3$, $c_0(v')=1$ and $c_0(v'')=2$ for $v', v'' \in V(B_v)$. By exchanging the colors of $v$ and $v'$, $c_0$ is still an optimal packing coloring of $G$, and we have analogous situation as above, and so Lemma \ref{lemma1} yields the result. 
Finally, suppose that both, $u$ and $v$, satisfy the property $(c3)$. Without loss of generality we may assume that $c_0(u)=3$. Then, by letting $c'(u)=c'(v)=1$, $c'(v')=c'(u')=2$, $c'(u'')=c'(v'')=3$ for $u', u'' \in V(B_u)$, $v', v'' \in V(B_v)$, and $c'(x)=c_0(x)$ for all $x \in V(G) \setminus \{u,v,u',u'',v',v''\}$, $c'$ is a packing coloring of $G-e$ using fewer colors than $c_0$. Namely, $|(c')^{-1}(1) \cup (c')^{-1}(2)|=k-1+|V(B)|-p_3$ and $|(c')^{-1}(3)|=2$. Thus, $\chi_\rho(G-e) < \chi_\rho(G)$. 

\textbf{Case c.2} $u \in V(B), v \notin V(B)$. \\
If $u$ satisfies the property $(c1)$ or $(c2)$, then without loss of generality we may assume that $c_0(v) \geq 3$. If there exists $x \in V(B)$, $x \neq u$, and also some $a \in V(B_x)$ with $c(a) \geq 3$, then using Lemma \ref{lemma1} we infer the result. Note that such $a$ certainly exists if $x$ satisfies the property $(c1)$ or $(c2)$. Otherwise, $x$ satisfies the property $(c3)$. By exchanging the colors $c_0(x) \geq 3$ and $c_0(x')=1, x' \in V(B_x)$, $c_0$ is still an optimal packing coloring of $G$ and again we have an analogous situation. Thus, the result follows.

If $u$ satisfies the property $(c3)$, then there exists $x \in V(B)$ having one of the properties $(c1)$ or $(c2)$. Note that $c_0(u) \geq 3$ and without loss of generality we may also assume that $c_0(v)=1$. By exchanging the colors of $u$ and $v$ and then using Lemma \ref{lemma1} for the vertices $v$ and $a \in V(B_x)$ with the color $c(a) \geq 3$, we infer that $\chi_\rho(G-e) < \chi_\rho(G)$.  

\textbf{Case c.3} $u \notin V(B), v \notin V(B)$. \\
Note that in this case $u,v \in V(B_z)$ for some $z \in V(B)$ with the property $(c1)$ or $(c2)$. Since there exists at least one vertex in $V(B_z)$, which receives color greater than $2$ by $c_0$, we may assume that $c_0(u) \geq 3$. Note that the color $c_0(u)$ appears only on $u$. Next, since at least one vertex from the same block as $u$ receives color $1$ by $c_0$, we can also assume that $c_0(v)=1$. 
Then, by letting $c'(u)=1$ and $c'(x)=c_0(x)$ for all $x \in V(G) \setminus \{u\}$, $c'$ is a packing coloring of $G-e$ using less colors than $c_0$. The proof is  complete.
\qed
\end{proof}

\section{Concluding remarks}

Critical graphs for the chromatic number are one of the classical topics in graph theory. By Dirac's theorem, $\delta(G)\ge k-1$ if $G$ is $k$-critical. As shown in this paper, there exist packing chromatic critical graphs with minimum degree $1$, and in fact, as proven in Section~\ref{sec:vertex}, there are $k$-$\pch$-critical trees for any $k\ge 2$. Although by Theorem~\ref{thm:diam2block}, $\delta(G)\ge 2$ in diameter $2$ block $\pch$-critical graphs $G$, already among $\pch$-critical block graphs with diameter $3$, there are several types of these graphs that contain leaves. 

It is also well known that the operation of join performed on two critical graphs for the chromatic number yields a critical graph. While an analogous statement for the packing chromatic number is again not true (a small example is the join of two copies of $P_4$, for which one can easily see that it is not $\pch$-critical), we wonder whether there is some natural graph operation that would preserve $\pch$-critical graphs. 

A natural problem that arises from Section~\ref{sec:vertex} is to characterize all $4$-$\pch$-critical graphs, which is done in that section only within two classes of graphs. We note that~\cite[Proposition 3.2]{goddard-2008} gives a characterization of 2-connected graphs $G$ with $\pch(G)=3$, which could be very useful in investigating $4$-$\pch$-critical graphs. 

Finally, we propose a problem of characterizing $\pch$-critical graphs with radius $2$. 


\section*{Acknowledgements}

The authors acknowledge the financial support from the Slovenian Research Agency (the research project J1-9109, B.B. also by the research core funding No.\ P1-0297  and J.F. by the research core funding No.\ P1-0403).


\end{document}